\documentclass{amsart}
\usepackage{amsmath,amsthm, amssymb, mathtools, xypic}
\usepackage{cite,  url,  mathscinet}
\usepackage[colorlinks]{hyperref}
\usepackage{enumitem}

 \usepackage{scalerel, stackengine}
 
\begin{document}

%macros
\renewcommand{\AA}{\mathbb{A}}

\renewcommand\widetilde[1]{\ThisStyle{%
  \setbox0=\hbox{$\SavedStyle#1$}%
  \stackengine{-.1\LMpt}{$\SavedStyle#1$}{%
    \stretchto{\scaleto{\SavedStyle\mkern.2mu\sim}{.5467\wd0}}{.7\ht0}%
%    .2mu is the kern imbalance when clipping white space
%    .5467++++ is \ht/[kerned \wd] aspect ratio for \sim glyph
  }{O}{c}{F}{T}{S}%
}}

\newcommand{\djunion}{\mathop{\bigsqcup}}
\newcommand{\BB}{\mathbb{B}}
\newcommand{\DD}{\mathbb{D}}
\newcommand{\KK}{\mathbb{K}}
\newcommand{\QQ}{\mathbb{Q}}
\newcommand{\ZZ}{\mathbb{Z}}
\newcommand{\FF}{\mathbb{F}}
\newcommand{\LL}{\mathbb{L}}
\newcommand{\NN}{\mathbb{N}}
\newcommand{\EE}{\mathbb{E}}
\newcommand{\HH}{\mathbb{H}}
\newcommand{\RR}{\mathbb{R}}
\newcommand{\PP}{\mathbb{P}}
\newcommand{\TT}{\mathbb{T}}
\renewcommand{\SS}{\mathbb{S}}
\newcommand{\CC}{\mathbb{C}}
\newcommand{\GG}{\mathbb{G}}
\newcommand{\WW}{\mathbb{W}}

\newcommand{\Oc}{\mathcal{O}}
\newcommand{\Ac}{\mathcal{A}}
\newcommand{\Cc}{\mathcal{C}}
\newcommand{\Ec}{\mathcal{E}}
\newcommand{\Fc}{\mathcal{F}}
\newcommand{\Sc}{\mathcal{S}}
\newcommand{\Gc}{\mathcal{G}}
\newcommand{\Hc}{\mathcal{H}}
\newcommand{\Pc}{\mathcal{P}}

\renewcommand{\Mc}{\mathcal{M}}
\newcommand{\Vc}{\mathcal{V}}

\newcommand{\coker}{\mathrm{coker}}
\newcommand{\gf}{\mathfrak{g}}
\newcommand{\pf}{\mathfrak{p}}
\newcommand{\mf}{\mathfrak{m}}
\newcommand{\qf}{\mathfrak{q}}
\newcommand{\kf}{\mathfrak{k}}
\newcommand{\hf}{\mathfrak{h}}
\newcommand{\Xf}{\mathfrak{X}}

\newcommand{\tensor}{\otimes}
\newcommand{\comptensor}{\widehat{\otimes}}

\newcommand{\cont}{\mathrm{cont}}
\newcommand{\Newt}{\mathrm{Newt}}

\newcommand{\actsonr}{\mathrel{\reflectbox{$\righttoleftarrow$}}}
\newcommand{\actsonl}{\mathrel{\reflectbox{$\lefttorightarrow$}}}

\newcommand{\isoeq}{\cong}
\newcommand{\cech}{\vee}
\newcommand{\dsum}{\mathop{\oplus}}

\newcommand{\Tr}{\mathrm{Tr}}
\newcommand{\Gal}{\mathrm{Gal}}

\newcommand{\Mod}{\mathrm{Mod}}
\newcommand{\Rep}{\mathrm{Rep}}
\newcommand{\Hecke}{\mathrm{Hecke}}
\newcommand{\Fil}{\mathrm{Fil}}
\newcommand{\Qlog}{\mathrm{Qlog}}
\newcommand{\Image}{\mathrm{Im}}
\newcommand{\Ker}{\mathrm{ker}}
\newcommand{\Id}{\mathrm{Id}}
\newcommand{\std}{\mathrm{std}}
\newcommand{\dR}{\mathrm{dR}}
\newcommand{\Gr}{\mathrm{Gr}}
\newcommand{\Aut}{\mathrm{Aut}}
\newcommand{\Pic}{\mathrm{Pic}}
\newcommand{\Div}{\mathrm{Div}}
\newcommand{\Frob}{\mathrm{Frob}}
\newcommand{\Proj}{\mathrm{Proj}}
\newcommand{\Spa}{\mathrm{Spa}}
\newcommand{\Spf}{\mathrm{Spf}}
\newcommand{\Spec}{\mathrm{Spec}}
\newcommand{\Hom}{\mathrm{Hom}}
\newcommand{\Isom}{\mathrm{Isom}}
\newcommand{\proet}{\mathrm{pro\acute{e}t}}
\newcommand{\Lie}{\mathrm{Lie}}
\newcommand{\perf}{\mathrm{perf}}
\renewcommand{\Hecke}{\mathrm{Hecke}}

\newcommand{\IC}{\mathrm{IC}}

\newcommand{\Map}{\mathrm{Map}}
\newcommand{\Ext}{\mathrm{Ext}}

\newcommand{\Homc}{\mathrm{Hom}_\mathrm{cont}}

\newcommand{\height}{\mathrm{ht}}
\newcommand{\un}{\mathrm{un}}
\newcommand{\et}{\mathrm{\acute{e}t}}
\newcommand{\profet}{\mathrm{prof\acute{e}t}}
\newcommand{\an}{\mathrm{an}}
\newcommand{\cris}{\mathrm{cris}}
\newcommand{\cyc}{\mathrm{cyc}}
\newcommand{\ad}{\mathrm{ad}}
\newcommand{\coh}{\mathrm{coh}}
\newcommand{\GL}{\mathrm{GL}}
\newcommand{\union}{\bigcup}

\newcommand{\rank}{\mathrm{rank}}
\newcommand{\ord}{\mathrm{ord}}
\newcommand{\Cont}{\mathrm{Cont}}
\newcommand{\Bun}{\mathrm{Bun}}
\newcommand{\piF}{\varpi_F}
\newcommand{\Irred}{\mathrm{Irred}}

\newcommand{\Ad}{\mathrm{Ad}}

\newcommand{\tr}{\mathrm{tr}}

\newcommand{\Conf}{\mathrm{Conf}}
\newcommand{\PConf}{\mathrm{PConf}}
\newcommand{\Sym}{\mathrm{Sym}}
\newcommand{\End}{\mathrm{End}}

\DeclarePairedDelimiter\floor{\lfloor}{\rfloor}

%theorem environments
\theoremstyle{plain}

%the following are labelled without section numbers
\newtheorem{maintheorem}{Theorem}
\renewcommand{\themaintheorem}{\Alph{maintheorem}} 
\newtheorem{maincorollary}[maintheorem]{Corollary}

%unlabelled 
\newtheorem*{theorem*}{Theorem}

%labeled with section numbers
\newtheorem{theorem}{Theorem}[section]
\newtheorem{conjecture}[theorem]{Conjecture}
\newtheorem{proposition}[theorem]{Proposition}
\newtheorem{lemma}[theorem]{Lemma}
\newtheorem{corollary}[theorem]{Corollary}

\theoremstyle{definition}

%labeled with section numbers
\newtheorem{example}[theorem]{Example}
\newtheorem{definition}[theorem]{Definition}
\newtheorem{remark}[theorem]{Remark}
\newtheorem{question}[theorem]{Question}

%unlabeled
\newtheorem*{principle*}{Principle}

\newcommand{\motcomp}{\widehat{\Mc_\LL}} 
\newcommand{\mot}{\Mc}
\newcommand{\Var}{\mathrm{Var}}
\newcommand{\van}{\mathrm{van}}
\newcommand{\HS}{\mathrm{HS}} 
\newcommand{\MSSP}{\mathrm{MSSP}}
\newcommand{\FinEt}{\mathrm{Fin\acute{E}t}}
\newcommand{\GVSH}{\mathrm{GVSH}}

\newcommand{\Pow}{\mathbf{Pow}}

\title[M.R.V. and rep. stability I: Configuration spaces]{Motivic random variables and representation stability I: Configuration spaces}
\author{Sean Howe}
\email{howe@math.utah.edu}
\address{Department of Mathematics, University of Utah, Salt Lake City, UT 84112}
%\classification{14G10, 18F30, 55R80}

\begin{abstract}
We prove a motivic stabilization result for the cohomology of the local systems on configuration spaces of varieties over $\CC$ attached to character polynomials. Our approach interprets the stabilization as a probabilistic phenomenon based on the asymptotic independence of certain \emph{motivic random variables}, and gives explicit universal formulas for the limits in terms of the exponents of a motivic Euler product for the Kapranov zeta function. The result can be thought of as a weak but explicit version of representation stability for the cohomology of ordered configuration spaces. In the sequel, we find similar stability results in spaces of smooth hypersurface sections, providing new examples to be investigated through the lens of representation stability for symmetric, symplectic and orthogonal groups. 
\end{abstract}

\maketitle

\tableofcontents

\section{Introduction}

In this paper, we bring together two recent threads in the theory of cohomological stability for configuration spaces: the theory of representation stability of Church \cite{Church-HomStab} and Church-Farb \cite{ChurchFarb-RepStab} and the motivic approach of Vakil-Wood \cite{VakilWood-Discriminants}. We show that the families of local systems on unordered configuration spaces of algebraic varieties studied in representation stability have natural motivic avatars in the Grothendieck ring of varieties, and that these stabilize under suitable motivic measures (cf. Theorem  \ref{thm:GeneralMotivicStab} below). A new and important feature of our approach is that we provide explicit formulas for the stable values using the language of probability theory adapted to a motivic setting (the \emph{motivic random variables} of the title). In particular, we find explicit universal formulas for the Hodge Euler characteristics of the stable cohomology of these local systems (cf. Theorem \ref{thm:VirtualLocHS} below). This includes, e.g., formulas for the Hodge Euler characteristics of the stable cohomology of generalized configuration spaces (cf. Corollary \ref{cor:RelativeStab} below). These results suggest that, at least in some cases, there should be a simple description of the stable cohomology of certain generalized
configuration spaces from which the stable cohomology of all local systems appearing in the theory representation stability for varieties can be deduced (cf. Remark \ref{remark:candidate-cohomology} below).  

	Over finite fields, the point-counting analogs of our results are due to Chen \cite{Chen-ConfPointCounting}. For motivic measures factoring through Chow motives, including the Hodge realization, our results can also be deduced from prior work of Getzler \cite{Getzler-MixedHodge}. We discuss both of these connections further in Remark \ref{rmk:ChenGetzler} and Appendix \ref{appendix:GetzlerChen}. 

 In the remainder of the introduction we describe our results in more detail in the case of the Hodge realization. 
 
\subsection{Motivic stabilization of character polynomials}\label{subsec:MotStabCharPoly}
Let $Y/\CC$ be a smooth connected quasi-projective variety. We denote by $\Conf^n (Y)$ the configuration space of unordered $n$-tuples of distinct points on $Y$. We denote by $\PConf^n (Y)$ the configuration space of ordered $n$-tuples of distinct points on $Y$, so~that 
\begin{equation}\label{eqn:SnCover} \PConf^n (Y) \rightarrow \Conf^n (Y) \end{equation}
is an $S_n$-cover (here $S_n$ denotes the symmetric group of permutations of $\{1,...,n\}$). 

The compactly supported cohomology of any local system $\Vc$ on $\Conf^n (Y)$ trivialized on this cover is naturally equipped with a polarizable mixed Hodge structure (cf. Subsection \ref{subsection:GVHS}). In particular, denoting by $W$ the weight filtration, 
\[ \Gr_W H^i_c(\Conf^n (Y), \Vc) \]
is a direct sum of polarizable Hodge structures. We denote 
\[ \QQ(-1) = H^2_c(\AA^1), \]
the Tate Hodge structure of weight 2, and $\QQ(n)=\QQ(-1)^{\otimes -n}$. 
  
We define $K_0(\HS)$ to be the Grothendieck ring of polarizable Hodge structures, which is the quotient of the free $\ZZ-$module with basis given by isomorphism classes $[V]$ of polarizable $\QQ-$Hodge structures $V$ by the relation $[V_1 \dsum V_2]=[V_1]+[V_2]$. It is a ring with $[V_1]\cdot[V_2]=[V_1 \otimes V_2]$.

For a local system $\Vc$ on $\Conf^n (Y)$ that becomes trivial on $\PConf^n (Y)$, we define 
\begin{equation}\label{eqn:LocSysChiHS} \chi_\HS(\Vc) = \sum -1^i \left[\Gr_W H^i_c(\Conf^n (Y), \Vc)\right] \in K_0(\HS). \end{equation}

The ring of character polynomials \cite{CEF-FIMod} is the countable polynomial ring $\QQ[X_1, X_2,...]$. For any $n$ we map the ring of character polynomials to the ring of rational class functions on $S_n$, $\Rep S_n \otimes \QQ$, by mapping $X_i$ to the function which counts cycles of length $i$. In particular, any character polynomial $p$ defines a sequence of ($\QQ$-)virtual representations $\pi_{p,n}$ of $S_n$ for increasing $n$. 

Thus, given a character polynomial $p$, we can use the cover (\ref{eqn:SnCover}) to define a sequence of ($\QQ$-)virtual $\QQ$-local systems $\Vc_{\pi_{p,n}}$ on $\Conf^n (Y$). Then, extending formula (\ref{eqn:LocSysChiHS}), the compactly supported cohomology of $\Vc_{\pi_{p,n}}$ gives a class $\chi_\HS(\Vc_{\pi_{p,n}})$ in the rationalized Grothendieck ring $K_0(\HS)_\QQ$. This ring has a natural completion $\widehat{K_0(\HS)_\QQ}$ with respect to the weight filtration, and one consequence of representation stability for configuration spaces as in \cite{Church-HomStab} is that for any fixed character polynomial~$p$, 
\begin{equation}\label{eqn:limitCharPolyHS} \lim_{n\rightarrow \infty} \frac{\chi_\HS(\Vc_{\pi_{p,n}})}{\QQ(-n \cdot \dim Y)} 
\end{equation}
exists in $\widehat{K_0(\HS)_\QQ}$. 

For a variety $X/\CC$, we denote 
\[ [X]_\HS = \sum_i (-1)^i \left[\Gr_W H^i_c(X(\CC),\QQ)\right] \in K_0(\HS). \]
Our main result, Theorem \ref{thm:VirtualLocHS} below, gives an explicit formula for the limit (\ref{eqn:limitCharPolyHS}) when we normalize by $[\Conf^n (Y)]_\HS$ instead of $\QQ(-n \cdot \dim Y)$. To state it efficiently, we introduce the map
\begin{align}
\nonumber \EE_\infty: \QQ[X_1,X_2,...] & \rightarrow \widehat{K_0(\HS)_\QQ} \\
\nonumber p & \mapsto  \EE_\infty[p]=\lim_{n\rightarrow \infty} \frac{\chi_\HS(\Vc_{\pi_{p,n}})}{[\Conf^n (Y)]_\HS}. 
\end{align}
$\EE_\infty$ should be thought of as a motivic version of the limit of the expectation functions of uniform probability measures on $\Conf^n (Y)$, in a way that will be made precise in Section \ref{sec:Probability}.

For a variety $Y/\CC$, the Hodge zeta function of $Y$ is
\[ Z_{Y,\HS}(t) = 1 + [Y]_\HS t + [\Sym^2 Y]_\HS t^2 +... \in 1 + K_0(\HS)[[t]]. \]
It admits a formal Euler product
\[ Z_{Y,\HS}(t) = \prod_{k \geq 1} \left(\frac{1}{1 + t^k}\right)^{M_k([Y]_\HS)} \]
where the $M_k([Y]_\HS)$ are elements of $K_0(\HS)_\QQ$  uniquely determined by this identity\footnote{We adopt this notation to match with the point-counting result \cite[Theorem 3 and Corollary~4]{Chen-ConfPointCounting}.}. The element $M_k([Y]_\HS)$ should be thought of as a motivic analog of the number of closed points of degree $k$ on a variety over a finite field (which appear as the corresponding exponents in Euler products for zeta functions of varieties over finite fields). 

\begin{maintheorem} \label{thm:VirtualLocHS} 
Let $Y/\CC$ be an $d$-dimensional smooth connected quasi-projective variety, and $\EE_\infty$, $X_k$, and $M_k([Y]_\HS)$ as defined above. For $t$ a formal variable,
\[ \EE_\infty \left[ (1+t)^{X_k} \right] = \left(1+ \frac{\QQ(dk)}{1+\QQ(dk)} t\right)^{M_k([Y]_\HS)} \]
and for $t_1, t_2, ...$ formal variables,
\[ \EE_\infty\left[ \prod_k (1+t_k)^{X_k} \right] = \prod_k \EE_\infty \left[ (1+t_k)^{X_k} \right] \]
where in both statements exponentiation of $(1+t)$ is understood in terms of the formal power series
\[ (1+t)^a = \exp (\log(1+t) \cdot a)) = \sum_{i=0}^\infty \binom{a}{i} t^i \]
and $\EE_\infty$ is applied to the coefficients of a power series. 
In particular, for any character polynomial $p$ we can obtain an explicit formula for $\EE_\infty [p]$ by expressing $p$ as a sum of monomials in the $X_k$ and using the formulas above. 
\end{maintheorem}

\begin{remark}\label{rmk:ChenGetzler} In fact, in Subsection \ref{subsec:MotAnalogCharPoly} we will construct avatars for the virtual local systems $\Vc_{\pi_{p,n}}$ living in a relative Grothendieck ring of varieties over $\Conf^n (Y)$, and then prove analogs of Theorem~\ref{thm:VirtualLocHS} and Corollary~\ref{cor:RelativeStab} below with $\HS$ replaced by an arbitrary motivic measure $\phi$ (under a technical hypothesis on $\phi$ and $Y$). These analogs appear as Theorem \ref{thm:GeneralMotivicStab} and Corollary \ref{cor:GeneralRelativeStab}. The reason it is possible to give motivic avatars for these virtual local systems is that they can be isolated by taking virtual sums of generalized configuration spaces; this corresponds roughly to the fact that permutation representations span the representation rings of symmetric groups. 

Theorem \ref{thm:VirtualLocHS} can also be deduced from results of Getzler~\cite{Getzler-MixedHodge}, as we show in Appendix~\ref{appendix:GetzlerChen}. However, there are cases of the more general Theorem \ref{thm:GeneralMotivicStab} which cannot be deduced from \cite{Getzler-MixedHodge} -- for example, when $\phi$ is the motivic measure with values in the Grothendieck ring of varieties completed for weight filtration, and $Y$ is a stably rational variety (cf. Example \ref{example:stably-rational}). Furthermore, the methods of our proof are of independent interest: the use of the power structure on the Grothendieck ring of varieties as a way of organizing arguments with configuration spaces modeled after point-counting arguments over finite fields is crucial in the sequel \cite{Howe-MRV2} where no analog of Getzler's result is available. The point-counting analog of Theorem \ref{thm:VirtualLocHS} is due to Chen \cite[Corollary 4]{Chen-ConfPointCounting}, and we also discuss this connection in Appendix \ref{appendix:GetzlerChen}. As is typical for this type of result (cf. \cite{VakilWood-Discriminants}), while the statements in the motivic and point counting settings are quite similar, neither implies the other, and the proofs in the motivic setting are more involved.   
\end{remark}

\begin{remark}
The probabilistic interpretation of this theorem is that the $X_k$ define asymptotically independent  \emph{motivic random variables} with asymptotic binomial distributions. The asymptotic binomial distributions are characterized by saying that $X_k$ converges in distribution to the ``sum of'' $M_k([Y]_\HS)$ Bernoulli random variables that are 1 with probability $\frac{\QQ(dk)}{1+\QQ(dk)}$ -- note that we interpret this purely as a statement about the moment generating functions (cf. Section \ref{sec:Probability}). The asymptotic independence and asymptotic distributions are natural in the analogous point-counting result over finite fields \cite[Corollary 4]{Chen-ConfPointCounting}: in that setting, $X_k$ corresponds to the random variable counting points of degree $k$ in a configuration, so that the $X_k$ can be described by summing up indicator random variables over closed points on the ambient variety $Y$. The asymptotic independence and asymptotic distributions of these indicator random variables have natural intuitive explanations. 
\end{remark}

\begin{example}\label{Example:ConfigSpaces}
For each $n\geq 1$, $X_1$ is the character of the permutation representation on the set $\{1,...,n\}$. Thus, denoting by $\Conf^{a \cdot b^{n-1}} (Y)$ the generalized configuration space of $n$ distinct points on $Y$, one labeled by $a$ and the remaining $n-1$ labeled by~$b$, we have
\[ \chi_\HS(\Vc_{\pi_{X_1,n}}) = [\Conf^{a \cdot b^{n-1}}(Y)]_\HS. \]
Theorem \ref{thm:VirtualLocHS} then gives
\[ \lim_{n\rightarrow \infty} \frac{[\Conf^{a \cdot b^{n-1}} (Y)]_\HS}{[\Conf^{n} (Y)]_\HS} = \left( \frac{\QQ(d)}{1+\QQ(d)}\right)[Y]_\HS \]
(here we use that $M_1([Y]_\HS)=[Y]_\HS$). 
\end{example}

\subsection{Motivic stabilization of generalized configuration spaces}
Example \ref{Example:ConfigSpaces} can be extended to all generalized configuration spaces. If ${\tau=a_1^{l_1}\cdot a_2^{l_2} \cdot ... \cdot a_m^{l_m}}$ is a generalized partition, we denote by $\Conf^{\tau} (Y)$ the configuration space of $|\tau|=\sum l_i$ distinct points on $Y$, with $l_i$ of the points labeled by $a_i$ for each $i$. We obtain

\begin{maincorollary}\label{cor:RelativeStab}
Notation as in Theorem \ref{thm:VirtualLocHS}, if $\tau = a_1^{l_1}...a_m^{l_m}$ is a generalized partition then 
\begin{multline*}
\lim_{n \rightarrow \infty} \frac{  [\Conf^{\tau \cdot *^{n-|\tau|}} (Y)]_\HS }{ [\Conf^n (Y)]_\HS} = \\
\frac{\partial_{t}^{\tau}|_{t_\bullet = 0}}{\tau!} \prod_k \left( 1 + \frac{\QQ(dk)}{1+\QQ(dk)} ( t_1^k + t_2^k + ... + t_m^k) \right)^{ M_k([Y]_\HS) }. 
\end{multline*}
\end{maincorollary}

\begin{remark}
For an orientable manifold $M$, cohomological stabilization for the sequence of spaces $\Conf^\tau \cdot *^{n-|\tau|} (M)$  was observed as a consequence of representation stability already by Church \cite[Theorem 5]{Church-HomStab}. The existence of the limit appearing in Corollary \ref{cor:RelativeStab} for an algebraic variety $Y$ was also shown by Vakil-Wood \cite{VakilWood-Discriminants}. Thus, the main contribution of Corollary \ref{cor:RelativeStab} is the explicit formula in terms of the $M_k([Y]_\HS)$. However, our proofs of Corollary \ref{cor:RelativeStab} and Theorem \ref{thm:VirtualLocHS} do not use the results of Church or Vakil-Wood as input, and thus also provide a new proof of the existence of this limit (which is closely related to, but distinct from, the proof given by Vakil-Wood).  
\end{remark}

\begin{remark}\label{remark:candidate-cohomology} For $\tau$ a generalized partition as in Corollary \ref{cor:RelativeStab}, the variety $\Conf^{\tau \cdot *^{n-|\tau|}} (Y)$ is a quotient of $\Conf^{\tau' \cdot *^{n-|\tau|}} (Y)$ where $\tau'$ is a partition of the same length as $\tau$ but with all labels distinct. Thus, the stable cohomology of any sequence of generalized configuration spaces appearing in Corollary \ref{cor:RelativeStab}, and in fact of any of the sequences of local systems coming from the theory of representation stability, can be computed in terms of the stable cohomology of these special families. 
	
In particular, our results put a large number of constraints on the stable cohomology of
\[ \Conf^{a_1 \cdot \ldots \cdot a_m \cdot *^{n-m}}(Y) \textrm{ as } n \rightarrow \infty, \]
in that they give a formula for the $S_{m}$-equivariant Hodge Euler characteristic rather than just the Hodge Euler characteristic. Moreover, the simple probabilistic formulas of Theorem \ref{thm:VirtualLocHS} lead to simple candidate stable cohomology groups compatible with these constraints, and it may be possible to prove these candidates are related to, or in some cases even equal to, the actual stable cohomology groups by using other known descriptions of the cohomology of configuration spaces (e.g. the explicit presentation of the Leray spectral sequence in \cite{totaro:configuration}). This strikes us as a promising avenue for future work, but we have not pursued it further at this time. 
\end{remark}

\subsection{Multiplicities of irreducible representations}
An interesting open problem in the theory of representation stability of configuration spaces is to compute the families of irreducible representations appearing in the stable cohomology of $H^i(\PConf^n (Y), \QQ)$ for $i$ a fixed degree (cf., e.g., the discussion around \cite[Corollary 1.6]{CEF-FIMod}). As in \cite{CEF-RepStabFinField} and \cite{Chen-ConfPointCounting}, we find here that it is simpler to instead fix a family of irreducible representations, corresponding by the theory of \cite{ChurchFarb-RepStab} to a partition $\tau$, and then to obtain information about the cohomological degrees where that specific family of representations appears stably. 

For a partition $\tau$, denote by $\pi_{\tau,n}$ the representation of $S_n$ corresponding to $\tau$, and by $\Vc_{\tau,n}$ the local system on $\Conf^n (Y)$ attached to $\pi_{\tau,n}$. The multiplicity of $\pi_{\tau,n}$ in $H^i(\PConf^n (Y), \QQ)$ is equal to the dimension of $H^i (\Conf^n (Y), \Vc_{\tau,n})$, and thus we are interested in computing the stabilization of the cohomology groups $H^i (\Conf^n (Y), \Vc_{\tau,n}).$ For $Y$ smooth, by Poincar\'{e} duality the same information is contained in the compact supported cohomology groups, and it will be technically more convenient for us to work with these. 

We denote
\[ \chi_\HS(\Vc_{\tau,n})= \sum -1^i \left[\Gr_W H^i_c (\Conf^n (Y),  \Vc_{\tau,n})\right] \in K_0(\HS). \]  
Then, using Theorem \ref{thm:VirtualLocHS}, we can compute 
\[ \lim_{n \rightarrow \infty} \frac{\chi_\HS(\Vc_{\tau,n})}{[\Conf^n (Y)]_\HS} \]
which gives us (via weights) some partial information about the cohomological degrees where $\tau$ can appear. Indeed, this limit is computed in Theorem \ref{thm:VirtualLocHS} when we take $p$ to be the character polynomial $s_\tau$ giving the character of $\pi_{\tau,n}$ for sufficiently large $n$ (which exists by results of \cite{CEF-FIMod}).

\begin{example}
Let $\tau$ be the partition $(1)$, which corresponds to the standard representation of $S_n$ for each $n$. The standard representation has character $s_{(1)} = X_1 - 1$. We find
\[ \lim_{n \rightarrow \infty} \frac{\chi_\HS(\Vc_{(1),n})}{[\Conf^n (Y)]_\HS} = \left(\frac{\QQ(d)}{1+\QQ(d)}\right)[Y]_\HS - 1.\]
In particular, the largest weight appearing in this formula is the negative of the smallest non-zero weight $k$ appearing in the cohomology of $Y$. Thus, by standard properties of weights (cf. \cite{Deligne-HodgeII}), if the standard representation appears stably in any degree $i < k/2$, it must appear stably again in another degree $i' < k$ with the opposite parity.  
\end{example} 

\subsection{Sketch of proofs of Theorem \ref{thm:VirtualLocHS} and Corollary \ref{cor:RelativeStab}}
We consider $\Conf^n (Y)$ as a moduli space parameterizing subvarieties of $n$ distinct points of $Y$, with universal family 
\[ Z_n \rightarrow \Conf^n (Y),\, Z_n := \{ (y,c) \in Y \times \Conf^n (Y)\, | \, y \in c \}. \]
We use the \emph{pre-$\lambda$} structure on the Grothendieck ring of varieties over $\Conf^n (Y)$ (cf. Section \ref{sec:GrothendieckRing}) to construct from $Z_n$ natural motivic avatars of the local systems appearing in Theorem \ref{thm:VirtualLocHS}. These are the \emph{motivic random variables} of the title. We note that, through the optic of the pre-$\lambda$ ring structure, the ring of character polynomials becomes the ring $\Lambda$ of symmetric functions \cite{Macdonald-SymmetricFunctions} playing its natural role in that theory. 

One would like to directly prove the asymptotic independence and binomial distributions for the motivic random variables corresponding to the $X_i$ as asserted in Theorem~\ref{thm:VirtualLocHS}, however, these are difficult to work with directly because for $i\neq 1$ they have no clear geometric meaning. Instead, we show that Theorem \ref{thm:VirtualLocHS} is actually equivalent to Corollary \ref{cor:RelativeStab}, then  prove Corollary \ref{cor:RelativeStab}. 

To show Theorem \ref{thm:VirtualLocHS} is equivalent to Corollary \ref{cor:RelativeStab}, we show that the character polynomials corresponding to generalized configuration spaces are a basis for the ring of character polynomials, and that on these elements the two descriptions of $\EE_\infty$ (one from Theorem \ref{thm:VirtualLocHS} and the other from Corollary \ref{cor:RelativeStab}) coincide.  

To prove Corollary \ref{cor:RelativeStab}, we use the geometric interpretation of the power structure on the Grothendieck ring of varieties due to Gusein-Zade, Luengo, and Melle-Hernandez \cite{GLM-PowerStructure}. A key step is our Lemma \ref{lem:EulerCoeffs}, which we use to find the $M_k([Y])$ as exponents in product expansions for generating functions of generalized configuration spaces. 

Although we work outside the setting of classical probability theory, the use of probabilistic language plays an important role in organizing our arguments and stating our theorems. In Section~\ref{sec:Probability} we develop the basics of algebraic probability theory which we will need in this paper and it sequel \cite{Howe-MRV2}. 

\subsection{Relation with the sequel}

In the sequel \cite{Howe-MRV2} we prove stabilization results analogous to Theorem \ref{thm:VirtualLocHS} and Corollary~\ref{cor:RelativeStab} for spaces of smooth hypersurface sections of a fixed smooth projective variety, as well as point-counting analogs. These give new geometric examples of representation stability for symmetric, orthogonal, and symplectic groups, in the sense that we find stabilization (in the Grothendieck ring of Hodge structures) of the cohomology of local systems corresponding to natural families of representations of these groups composed with the monodromy representation on the cohomology of the universal smooth hypersurface section. Combined with the results of the present paper, we view the results of \cite{Howe-MRV2} as strong evidence that one should also seek richer representation stability-type phenomena in the setting of smooth hypersurface sections. 

\subsection{Outline}
In Section \ref{sec:PreLambda} we recall the notion of a pre-$\lambda$ structure and the associated power structure as in \cite{GLM-PowerStructure}. In Section \ref{sec:GrothendieckRing} we recall some Grothendieck rings of varieties and the ${\textrm{pre-}\lambda}$ and power structures on them defined by the Kapranov zeta function, along with their geometric interpretations. In Section \ref{sec:Probability} we develop the notion of an algebraic probability measure. Finally, in Section \ref{sec:Proof} we prove Theorem~\ref{thm:VirtualLocHS} and Corollary~\ref{cor:RelativeStab}. 

In Appendix \ref{appendix:GetzlerChen}, we elaborate on the connection between our work and that of Chen \cite{Chen-ConfPointCounting} and Getzler \cite{Getzler-MixedHodge} as indicated in Remark \ref{rmk:ChenGetzler}.

\subsection{Notation}
For partitions we follow the conventions of Vakil-Wood \cite{VakilWood-Discriminants}, though we tend to avoid the use of $\lambda$ to signify a partition to avoid conflicts with the theory of pre-$\lambda$ rings. 

A variety over a field $\KK$ is a reduced finite-type scheme over $\KK$ (in particular, we do not require our varieties to be irreducible). It is quasi-projective if it can be embedded as a locally closed subvariety of $\PP^n_\KK$.  

For $Y \rightarrow S$ a map of quasi-projective varieties over a field $\KK$ and $\tau$ a partition, we write $\Conf^\tau(Y/S)$ for the relative (i.e., fiberwise) configuration space of $\tau$-labeled distinct points on $Y/S$. Concretely, if the multiplicities of $\tau$ are $(l_1, \ldots, l_m)$, then
\[ \Conf^\tau(Y/S) = \left( \underbrace{Y \times_S Y \times_S \ldots \times_S Y}_{l_1 + \ldots + l_m} \backslash \Delta \right) / S_{l_1} \times S_{l_2} \times \ldots \times S_{l_m}, \]
where $\Delta$ is the big diagonal (the locus where not all coordinates are distinct), and the product of symmetric groups $S_{l_i}$ acts by permuting the coordinates in the natural way. We note that $\Conf^\tau(Y/S)$ is a quasi-projective variety over $\KK$ with a natural structure map to $S$, thus it makes sense to write $\Conf^\tau(Y/S)/S$. When there is no risk of confusion, we will also write $\Conf^\tau(Y)$ for $\Conf^\tau(Y/\KK)$. Following the conventions used in the introduction, we will sometimes write $\Conf^n$ for $\Conf^\tau$ where $\tau$ is any partition with a single label of multiplicity $n$, and $\PConf^n$ for $\Conf^\tau$ where $\tau$ is any partition with $n$ distinct labels each appearing with multiplicity one.

Our notation for pre-$\lambda$ rings and power structures is introduced in Section \ref{sec:PreLambda}. We highlight the following point here: if $f \in 1+(t_1,t_2,...)R[[t_1, t_2,..]]$, then 
$f^r$ will always denote the naive exponential power series
\[ \exp \left( r\cdot \log f \right) \in   1+(t_1,t_2,...)R_\QQ[[t_1, t_2,..]]. \]
If $R$ is a pre-$\lambda$ ring then we denote an exponential taken in the associated power structure by $f^{_\Pow r}.$

Our notation for Grothendieck rings of varieties and motivic measures is explained in Section \ref{sec:GrothendieckRing}.

\subsection{Acknowledgements}
We thank Weiyan Chen, Matt Emerton, Benson Farb, Jesse Wolfson, and  Melanie Matchett Wood for helpful conversations. We thank Weiyan Chen, Matt Emerton, Benson Farb, Melanie Matchett Wood, and two anonymous referees for helpful comments on earlier drafts. This material is based upon work supported by the National Science Foundation under Award
No. DMS-1704005.

\section{Power structures and pre-$\lambda$ rings}\label{sec:PreLambda}
In this section, we recall the notion of a pre-$\lambda$ structure and the associated power structure as in \cite{GLM-PowerStructure}. Our only new contribution is Lemma~\ref{lem:EulerCoeffs}. 

\subsection{Symmetric functions}\label{subsec:SymFunc}
We denote by 
\[ \Lambda = \varprojlim_n \ZZ[t_1,...,t_n]^{S_n} \]
the graded ring of symmetric functions \cite{Macdonald-SymmetricFunctions}. Here the limit is of graded rings along the maps induced from
\begin{align}
\nonumber \ZZ[t_1,...,t_n] & \rightarrow  \ZZ[t_1,...,t_{n-1}] \\
\nonumber t_i & \mapsto  \begin{cases} t_i & 1 \leq i \leq {n-1} \\ 0 & i = n. \end{cases}
\end{align}

We define the \emph{complete symmetric functions}
\[ h_k := \sum_{(l_1, l_2,...) \; | \; \sum l_i = k} t_1^{l_1}t_2^{l_2}...  \]
the \emph{elementary symmetric functions}
\[ e_k := \sum_{i_1 \neq i_2 \neq... \neq i_n} t_{i_1} t_{i_2} \cdot ... \cdot t_{i_n} \in \Lambda, \]
and the \emph{power sum symmetric functions}
\[ p_k := \sum_i t_i^k \in \Lambda. \]
We have 
\[ \Lambda = \ZZ[h_1, h_2,...] = \ZZ[e_1, e_2, ...], \]
and
\[ \Lambda_\QQ = \QQ[p_1,p_2,...]. \]

For a partition $\tau=a_1^{l_1} a_2^{l_2}...$ we also define 
\begin{align}
\nonumber h_\tau & =  h_{l_1} \cdot h_{l_2} \cdot ... \\
\nonumber p_\tau & =  p_{l_1} \cdot p_{l_2} \cdot ... 
\end{align}
(which depends only on the multiplicity partition $m(\tau)$). 

We also define the \emph{Mobius-inverted power sum symmetric functions}
\[ p_k' := \frac{1}{k} \sum_{d|k} \mu(k/d) p_k \in \Lambda_\QQ.\]
We have
\[ p_k = \sum_{d|k} d p_d'. \]

The complete symmetric functions are related to the power sum symmetric functions by the identity
\begin{equation}\label{eqn:complete-powersum} d\log \sum h_k t^k = \sum_{k=0}^\infty p_{k+1} t^{k},
\end{equation}
and to the Mobius-inverted power sums by the Euler-product identity
\begin{equation}\label{eqn:complete-eulerprod} \sum_{k} h_k t^k = \prod_k \left(\frac{1}{1-t^k}\right)^{p_k'} 
\end{equation}
(which follows from (\ref{eqn:complete-powersum}) after taking $\log$ of both sides).

\subsection{Pre-$\lambda$ rings}\label{subsec:pre-lambda}
\begin{definition}
A \emph{pre-$\lambda$ ring} is a ring $R$ equipped with a group homomorphism
\begin{align}
\nonumber \sigma_t: (R, +) & \rightarrow  (1 + t R[[t]], \cdot) \\
\nonumber r& \mapsto  1 + \sigma_1(r) t + \sigma_2(r) t^2 + ... 
\end{align}
such that $\sigma_1(r) = r$, \textbf{and such that 
\[ \mathbf{\sigma_t(1) = \frac{1}{1-t}}\]
(i.e., in terms of the coefficients $\sigma_k$, $\sigma_k(1)=1$ for all $k$). }
\end{definition}

The condition on $\sigma_t(1)$ is not standard, but is natural in our context. In particular, any $\lambda$-ring satisfies this condition. 

\begin{example}\label{example:Z}
$\ZZ$ is a pre-$\lambda$ ring with $\sigma_t(n)=\left(\frac{1}{1-t}\right)^n$. 
\end{example}

\begin{example}{\label{example:naivePreLambda}}
Any $\QQ$-algebra $R$ is a pre-$\lambda$ ring with 
\[ \sigma_t(r)= \exp \left( \log \left(\frac{1}{1-t}\right) \cdot r \right) \]
where $\exp$ and $\log$ are defined using the standard power series expansions. 
\end{example}

\begin{example}\label{example:RepG}
For $G$ an algebraic group, the representation ring $\Rep G$ is a pre-$\lambda$ ring with $\sigma$-operations given by $\sigma_k([V])=[\Sym^k V]$. Note that Example \ref{example:Z} is also of this form for $G=\{e\}$ the trivial group. 
\end{example}

\newcommand{\Set}{\mathrm{Set}}
\begin{example}\label{example:PermG}
For $G$ a group we define the Grothendieck ring $K_0(G-\Set)$ of finite $G$-sets: It is spanned by the isomorphism classes $[X]$ of finite $G$-sets $X$, modulo the relation $[X \sqcup X']=[X]+[X']$. It is a ring under cartesian product $[X]\cdot[X']=[X\times X']$. It is equipped with a pre-$\lambda$ structure with $\sigma$-operations 
\[ \sigma_k (X) = X^k / S_k. \]
For any $G$ there is a natural map from $K_0(G-\Set)$ to $K_0(\Rep G)$ sending a set to the free vector space on that set, and this is a map of pre-$\lambda$ rings (with the pre-$\lambda$ structure on $\Rep G$ as in Example \ref{example:RepG}). 
\end{example}

We define functorial \emph{$\sigma$-operations} on pre-$\lambda$ rings $R$ by 
\[ r \mapsto \sigma_k(r). \]
More generally, we define a functorial set-theoretic pairing 
\[ (\;,\;): \Lambda \times R \rightarrow R \]
on pre-$\lambda$ rings $R$ by requiring that for any $r \in R$, 
\[ (\; , r) : \Lambda \rightarrow R \]
is the unique ring homomorphism sending $h_k$ to $\sigma_k(r)$.
 
Using this pairing, any element $f\in \Lambda$ defines a functorial operation on pre-$\lambda$ rings by 
\[ r \mapsto (f, r). \] 

In particular, we also define the \emph{$\lambda$-operations} by
\[ \lambda_k(r) := (e_k, r) \]
where the $e_k$ are the elementary symmetric functions. Classically pre-$\lambda$ rings and $\lambda$-rings are axiomatized using the $\lambda$-operations, however, for us it is more natural to use $\sigma$-operations.

For a typical $f \in \Lambda$,
\[ (f, \; ) : R \rightarrow R \]
is not additive, however, the \emph{Adams operations} 
\[ r \mapsto (p_k, r) \]
corresponding to the power sum functions $p_k$ give functorial endomorphisms of $(R,+)$ (this follows from the identity (\ref{eqn:complete-powersum})). As a consequence, the maps 
\begin{align}
\nonumber R & \rightarrow  R_\QQ \\
\nonumber r & \mapsto  (p_k', r)
\end{align}
are also additive maps, functorial in pre-$\lambda$ rings. By equation (\ref{eqn:complete-eulerprod}), they satisfy the Euler product identity in $R_\QQ$, 
\begin{equation}\label{prelambda-eulerprod}
\sigma_t(r) = \prod_k \left(\frac{1}{1-t^k} \right)^{(p_k', r)}.
\end{equation}

\begin{remark}
We will not need the fact that $(p_k', \;)$ is additive in the present work, however, it will play an important role in the sequel \cite{Howe-MRV2}. 
\end{remark}

\subsection{Power structures.}
The notion of a power structure, due to Gusein-Zade, Luengo, and Melle-Hernandez \cite{GLM-PowerStructure}, gives us a way to interpret expressions of the form 
\[f(t_1,t_2,..)^r \]
for $f(t) \in 1 + (t_1,t_2,...)R[[t_1, t_2,...]]$ and $r \in R$ for any pre-$\lambda$ ring $R$, satisfying many of the expected properties of exponentiation and functorial in the pre-$\lambda$ ring $R$. To avoid confusion with the naive exponential
\begin{align} 
\nonumber R \times 1+(t_1,t_2,...)R[[t_1,t_2,...]] & \rightarrow  1+(t_1,t_2,...)R_\QQ[[t_1,t_2,...]]  \\
\nonumber (r,f) & \mapsto  f^r := \exp\left(\log(f) \cdot r \right)
\end{align}
(which we use extensively), we will denote the power structure operation
\[ R \times 1+(t_1,t_2,...)R[[t_1,t_2,...]] \rightarrow 1+(t_1,t_2,...)R[[t_1,t_2,...]] \]
by
\[ (r, f) \mapsto f^{_\Pow r}. \]

Before explaining how the power structure is obtained from the pre-$\lambda$ structure, we first give an illustrative example to show how it differs from the naive exponential:
\begin{example}
With notation as in Example \ref{example:PermG}, for $X$ a finite $G$-set, one can show that
\[ (1+t)^{[X]} = \sum_{k=0}^\infty \left( \frac{\left[X^k \backslash \Delta_k \right]}{k!}\right) t^k \in K_0(G-\mathrm{Set})_\QQ, \]
where $\Delta_k$ denotes the big diagonal (i.e. the locus where the coordinates of $X^k$ are not all distinct).
On the other hand, for the power structure attached to the pre-$\lambda$ structure described in Example \ref{example:PermG}, one finds
\[ (1+t)^{_\Pow [X]} = \sum_{k=0}^\infty \left[(X^k \backslash \Delta_k) / S_k \right] t^k \in K_0(G-\mathrm{Set}), \]
where $S_k$ is the symmetric group acting by permutation of the coordinates.

Thus, the division by $k!$ in the naive exponential is replaced by a free quotient by the symmetric group $S_k$ in the power structure. If $G$ is the trivial group, then we are working in $K_0(\mathrm{Fin-Set})=\ZZ$, where the class of any set is given by its cardinality. In that case, the two formulas agree; for general $G$, however, they are not the same! 
\end{example}

The power structure operation is constructed by first defining, for a monomial $t^I$, 
\[ \left(\frac{1}{1-t^I}\right)^{_\Pow r} := \sigma_{t^I}(r). \]
To define the operation for an arbitrary power series, we first note that because 
\[ \sigma_{t^I}(r) = 1 + r\cdot t^I + ..., \]
any element 
\[ f \in 1+(t_1,t_2,...)R[[t_1,t_2,...]] \]
has a unique expression as an Euler product (for the power structure)
\[ f = \prod_I \left(\frac{1}{1-t^I}\right)^{_\Pow a_I}. \]
Indeed, this can be shown by induction on the multi-indices $I$ (under any ordering where the total exponent $\sum I$ is non-decreasing), with the $I$th factor chosen to ensure that the product has the desired coefficient for $t^I$. 

Thus, using this Euler product for $f$, we may define its powers by 
\[ f^{_\Pow r} := \prod_I \left(\frac{1}{1-t^I}\right)^{_\Pow a_I \cdot r}. \]
It satisfies the following properties \cite{GLM-PowerStructure}:
\begin{itemize}
\item $f ^{_\Pow 0} = 1$
\item $f ^{_\Pow 1} = f$
\item $(f \cdot g)^{_\Pow r} = f^{_\Pow r} \cdot g^{_\Pow r}$
\item $f^{_\Pow r_1 + r_2} = f^{_\Pow r_1} f^{_\Pow r_2}$
\item $f^{_\Pow r_1 r_2} = (f^{_\Pow r_2})^{_\Pow r_1}$.
\end{itemize}

We have the following useful lemma: 

\begin{lemma}\label{lem:EulerCoeffs}
Let $R$ be a pre-$\lambda$ ring, with associated power structure as described above. If $f \in 1 + (t_1,t_2,...) \ZZ[[t_1,t_2,...]]$ and  $r \in R,$ then
\[ f^{_\Pow r} = \prod_{k\geq 1} f(t_1^k,t_2^k,...)^{ (p_{k}',r)}, \]
where the exponentiation on the right hand side of the equality is the naive exponentiation and the identity is of elements in
$1 + (t_1,t_2,...) R_\QQ[[t_1,t_2,...]].$
\end{lemma}
\begin{proof}
Because the coefficients of $f$ are in $\ZZ$, it has an Euler product
\[ f = \prod_I \left( \frac{1}{1-t^I} \right)^{a_I} = \prod_I \left( \frac{1}{1-t^I} \right)^{_\Pow a_I}\]
with $a_I \in \ZZ$, which is independent of the power structure because it only involves integral powers. 
Thus, for any $r$, 
\begin{multline*}
f^{_\Pow r} =  \prod_I \left(\frac{1}{1-t^I}\right)^{_\Pow a_I \cdot r} = \prod_I \sigma_{t^I}(a_I \cdot r) = \prod_I \sigma_{t^I}(r)^{a_I} \\
 =  \prod_I \prod_k \left(\frac{1}{1-(t^{I})^k}\right)^{ a_I \cdot (p_{k}', r) } =  \prod_k \prod_I \left(\frac{1}{1-(t^{I})^k}\right)^{a_I \cdot (p_{k}', r) } \\
 = \prod_k \left( \prod_I \left(\frac{1}{1-(t^{I})^k}\right)^{a_I} \right) ^{(p_{k}', r)} = \prod_k f(t_1^k, t_2^k, ...)^{ (p_{k}', r) }. 
\end{multline*}
Here the step from the first to second line follows from equation (\ref{prelambda-eulerprod}).
\end{proof}

\section{Some Grothendieck rings of varieties}\label{sec:GrothendieckRing}
In this section, we define some Grothendieck rings of varieties and discuss the pre-$\lambda$ and power structures on them defined by the Kapranov zeta function. We then recall the geometric description of the associated power structure in characteristic zero given in \cite{GLM-PowerStructure}, and explain how to generalize it to perfect fields following the generalization of the pre-$\lambda$ structure in \cite{Mustata-ZetaBook}. We finish by explaining how to use these structures to understand configuration spaces. The key result of this section is Lemma \ref{lem:conf-Euler}, which establishes the fundamental relationship in $\Lambda$ which we use to prove Corollary~\ref{cor:RelativeStab}.

\label{subsec:GrothendieckRing}
Let $\KK$ be a a field. The \emph{Grothendieck ring of varieties} over $\KK$, $K_0(\Var/\KK)$, 
is generated by the isomorphism classes $[Y]$ of varieties over $Y/\KK$, modulo the relations 
\[ [Y]=[Y\backslash Z] + [Z] \]
for $Z$ a closed subvariety of $Y$. It is a ring with 
\[ [Y_1]\cdot[Y_2]=[Y_1 \times Y_2].\]
We refer to \cite[Chapter 7]{Mustata-ZetaBook}, as our basic reference for $K_0(\Var/\KK)$. 

We denote $\LL=\AA^1$, and 
\[ \Mc_\LL = K_0(\Var / \KK)[\LL^{-1}]. \]
The ring $\Mc_\LL$ has a decreasing dimension filtration, where $\Fil^i$ is generated by classes $[Y]/\LL^{m}$ where $[Y]$ has dimension $\leq m-i$. We denote the completion with respection to the dimension filtration by $\widehat{\Mc_\LL}.$

A motivic measure is a map of rings
\[ \phi: K_0(\Var / \KK) \rightarrow A. \]

\begin{example}\label{example:mot-measure}\hfill 
\begin{itemize}
\item If $\KK=\FF_q$, there is a point-counting measure 
\begin{align}
\nonumber \phi_q : K_0(\Var/\FF_q) & \rightarrow  \ZZ \\
\nonumber \; [X] & \mapsto  \# X(\FF_q). 
\end{align}
It extends to a map
\[ \phi_q: \Mc_\LL \rightarrow \ZZ[q^{-1}], \]
but is not continuous with respect to the dimension filtration, so does not extend to $\widehat{\Mc_\LL}$. 
\item If $\KK=\CC$, there is a Hodge measure
\begin{align}
\nonumber K_0(\Var/\CC) & \rightarrow  K_0(\HS) \\
\nonumber [X] & \mapsto [X]_\HS := \sum (-1)^i \left[\Gr_W H_c^i(X(\CC),\QQ) \right]
\end{align}
where $X$ is a quasi-projective variety over $\CC$, $W$ is the weight filtration, and  $K_0(\HS)$ is the Grothendieck ring of polarizable Hodge structures over $\QQ$. It extends naturally to $\Mc_\LL$, and to a map
\[ \motcomp \rightarrow \widehat{K_0(\HS)} \]
where the completion on the right is with respect to the weight grading. 
\end{itemize}
\end{example}

\subsection{Pre-$\lambda$ and power structure}\label{subsec:GrothPow}\hfill\\
\textbf{\emph{For this subsection, we assume 
$\KK$ is a perfect field}}. \\

\noindent We consider the ring\footnote{We use this modification of the Grothendieck ring to avoid some technicalities involving the possible existence of different pre-$\lambda$ structures on $K_0(\Var/\KK)$ when $\mathrm{char} \KK \neq 0$, depending on whether symmetric powers are defined using the scheme theoretic or stacky quotient (cf. \cite{ekedahl:geometric-invariant} for details).} 
\[ \widetilde{K_0}(\Var / \KK), \]
as defined in \cite[7.2]{Mustata-ZetaBook}.  It is the quotient of $K_0(\Var/\KK)$ by the relations  $[Y_1]=[Y_2]$ whenever there is a radicial surjective morphism $Y_1 \rightarrow Y_2$.

By \cite[Proposition 7.25]{Mustata-ZetaBook}, for $\KK$ of characteristic zero 
\[ K_0(\Var / \KK) = \widetilde{K_0}(\Var / \KK), \]
and by \cite[Proposition 7.26]{Mustata-ZetaBook}, for $\KK$ a finite field the point counting measures factor through $\widetilde{K_0}(\Var / \KK)$. 

By \cite[Proposition 7.28]{Mustata-ZetaBook}, $\widetilde{K_0}(\Var / \KK)$ has a pre-$\lambda$ structure such that for any quasi-projective variety $Y/\KK$, $\sigma_t([Y])$ is equal to the Kapranov zeta function:
 \[ \sigma_t([Y]) = Z_Y(t) := \sum_k [\Sym^k Y ] t^k. \]

In \cite{GLM-PowerStructure}, it is shown that for $\KK=\CC$ the corresponding power structure admits the following description on effective power series: If  
\[ f(t_1,...) = \sum_I [A_I] t^I \]
for quasi-projective varieties $A_I / \KK$, and $Y$ is a quasi-projective variety, then the coefficient of $t^{I'}$ in 
\[ f(t_1,...)^{_\Pow [Y]} \]
is the class of the variety parameterizing labelings of finite sets of distincts points of $Y$ by labels in
$\sqcup_{I} A_I$ with total weight $I'$ (where a label in $A_I$ has weight $I$). Concretely, 
\begin{multline*} f(t_1,...)^{_\Pow [Y]} = \\
\sum_I  \left [ \bigsqcup_{\substack{\{I_1, ..., I_n, m_1, ... m_n\} \\  \textrm{ s.t. } m_1 \cdot I_1+...+ m_n \cdot I_n = I}}  \left(\left( \prod_{k=1}^n  (A_{I_k} \times Y)^{m_k}  \right) \backslash \Delta \right)/ S_{m_1}\times...\times S_{m_n} \right] t^I 
\end{multline*}
where $\Delta$ is the big diagonal and the symmetric groups act by simultaneous permutation of points (i.e. coordinates in $Y$) and labels (i.e. coordinates in $A_{I_k}$). 

In fact, with the modification of working in $\widetilde{K_0}(\Var / \KK)$, this description is valid over any perfect field -- the modifications of the proof necessary are essentially the same as those used in \cite[Proposition 7.28]{Mustata-ZetaBook} to show that the Kapranov zeta function gives a pre-$\lambda$ structure in this setting. 

\begin{remark}
For the reader uncomfortable with the extension of the geometric description of the power structure outside of characteristic zero, we note that we use it only in the proof of Lemma \ref{lem:conf-Euler}, where the necessary point counting result can be verified directly. 
\end{remark}

\begin{example}\label{example:ConfGen}
Using the geometric description of the power structure we obtain a generating function for generalized configuration spaces (in $\widetilde{K_0}(\Var / \KK)$)
\[ (1+ t_1 + t_2 + ...)^{[Y]} = \sum_{I=(i_1,i_2,..)} [\Conf^{a_1^{i_1} \cdot a_2^{i_2} ...} (Y)] t_1^{i_1}t_2^{i_2}.... \]
\end{example}

\begin{lemma}\label{lem:ZeroPointCount}
Let $f \in \Lambda_\QQ$. If for any $q$ and any quasi-projective variety $Y / \FF_q$, 
\[ \phi_q((f, [Y])) =0, \]
then $f = 0$ (here $\phi_q$ is the point-counting measure).   
\end{lemma}
\begin{proof}
Since the basis $h_\tau$ maps under $(f, \;)$ to the monomials in the symmetric powers, this follows from \cite[Lemma 3.18]{VakilWood-Discriminants}. 
\end{proof}

For a partition $\tau$, we define $c_\tau \in \Lambda$ to be the unique element of $\Lambda$ such that for any quasi-projective variety $Y$ over any $\KK$, 
\[ (c_\tau, [Y]) = [\Conf^\tau (Y)]. \]
By \cite[3.19]{VakilWood-Discriminants} such a $c_\tau$ exists, is unique, and depends only on $m(\tau)$. Moreover the $c_\tau$ over all possible multiplicities form a basis for $\Lambda$ (this follows from the explicit formula in terms of $h_\tau$ given in \cite[3.19]{VakilWood-Discriminants}). For $I=(l_1,l_2,..)$ a sequence in $\ZZ_{\geq 0}$ that is eventually 0, we denote by $c_I$ the element $c_{\tau_I}$ for $\tau_I$ any partition with multiplicities given by $I$ up to reordering.  

\begin{lemma}\label{lem:conf-Euler}
\[ \prod_{k} (1 + t_1^k + t_2^k + ... )^{p_k'} = \sum_{I} c_I t^{I}. \]
\end{lemma}
\begin{proof}
For any quasi-projective $Y / \FF_q$, we apply $(\;,[Y])$ to the left-hand side to obtain  
\begin{align*}
\prod_{k} (1 + t_1^k + t_2^k + ... )^{(p_k', [Y])} & =  (1+t_1 + t_2 + ...)^{_\Pow [Y]} \\
& =  \sum_{I} [\Conf^{\tau_I} (Y)] t^I\\
& = \sum_{I} (c_I, [Y]) t^I,
\end{align*}
where the first line follows from Lemma \ref{lem:EulerCoeffs} and the second from Example \ref{example:ConfGen}. The result then follows from Lemma \ref{lem:ZeroPointCount} after applying the point-counting realization. 
\end{proof}

\subsection{Relative Grothendieck rings}\label{subsec:RelGrothRings}
For $\KK$ a field and $S/\KK$ a variety, we will also consider the relative Grothendieck ring 
\[ K_0(\Var / S) \]
of quasi-projective varieties over $S$, generated by isomorphism classes $[Y/S]$ of quasi-projective varieties $Y/S$ (i.e. maps $Y \rightarrow S$ of quasi-projective varieties) modulo the relations 
\[ [Y/S]=[Z/S]+[Y\backslash Z / S] \]
for any closed subvariety $Z \subset Y$ in a quasi-projective $Y/S$. It is a ring with operation
\[ [Y_1/S]\cdot[Y_2/S]=[Y_1 \times_S Y_2 / S]. \]

$K_0(\Var / S)$ is naturally an algebra over $K_0(\Var / \KK)$, and there is a natural map of $K_0(\Var / \KK)$-modules given by ``forgetting the structure morphism'':
\begin{align*}
K_0(\Var / S ) & \rightarrow  K_0(\Var / \KK) \\
\; [Y/S] & \mapsto  [Y].
\end{align*}

If $\KK$ is of characteristic zero then there is a relative Kapranov zeta function whose coefficients are relative symmetric powers over $S$ and which induces a pre-$\lambda$ and power structure on $K_0(\Var / S)$ (cf. \cite[Remark at the end of Section 1]{GLM-PowerStructure}).  Note that the forgetful map does not respect the ring or pre-$\lambda$ structures.

The explicit formula for a configuration space in terms of symmetric powers of \cite[3.19]{VakilWood-Discriminants} still holds in the relative setting, and in particular we deduce that for any $\tau$,
\[ (c_\tau, [Y/S]) = [\Conf^\tau (Y / S) / S]. \]

\subsection{Geometric variations of Hodge structure}\label{subsection:GVHS}
For $X$ a smooth connected variety over $\CC$, we denote by denote $\GVSH/X$ (sic) the category of geometric variations of Hodge structure on $X(\CC)$ as in \cite{Arapura-LerayMotivic}. For $\Vc$ an element of $\GVSH/X$, the theory of \cite{Arapura-LerayMotivic} produces a mixed Hodge structure on $H^i_c(X(\CC), \Vc)$. This induces a map of rings
\begin{align*}
\chi_\HS : K_0( \GVSH / X) & \rightarrow  K_0(\HS) \\
\;[\Vc] & \mapsto  \sum_i (-1)^i \left[\Gr_W H^i_c( X(\CC), \Vc) \right]. 
\end{align*}

We now discuss a compatibility between this construction and the assignment ${[Z] \mapsto [Z]_\HS}$ for $Z/X$ finite \'{e}tale. 

For $X/\CC$ as above, the category $\FinEt / X$ is equivalent to the category of finite $\pi_1(X(\CC))$-sets. We define $K_0(\FinEt/X)$ to be the Grothendieck ring of this category as defined in Example \ref{example:PermG}. There is a natural map of pre-$\lambda$ rings 
\begin{align*}
K_0(\FinEt/X) & \rightarrow K_0(\Var / X) \\
\;[Z / X] & \mapsto [Z / X].
\end{align*}
There is also a natural map of pre-$\lambda$ rings
\begin{align*}
 K_0(\FinEt / X) & \rightarrow K_0(\GVSH/X)\\
\;[f: Z / X] & \mapsto [f_*\QQ]. 
\end{align*}

\begin{lemma}\label{lem:Leray}
The diagram
\[ \xymatrix{ 
 & K_0(\Var/X) \ar[dr]^{\;\;\; [Z/X] \mapsto [Z]_\HS} & \\ 
K_0(\FinEt/X) \ar[ur]\ar[dr]& & K_0(\HS) \\
 & K_0(\GVSH/X) \ar[ur]_{\;\;\; \Vc \mapsto \chi_\HS(\Vc)} & 
} \]
commutes. 
\end{lemma}
\begin{proof}
This follows from compatibility with the Leray spectral sequence in \cite{Arapura-LerayMotivic}.
\end{proof}

\begin{remark}
By the same argument, for any $f: Z \rightarrow X$ smooth and proper we have 
\[ \chi_\HS ([Rf_* \QQ]) = [Z]_\HS. \]
\end{remark}

\section{Algebraic probability theory}\label{sec:Probability}
In this section we develop the notion of an algebraic probability measure. For our applications the key concept we must define is that of asymptotic independence, and we build up only the minimal amount of theory necessary in order to accomplish this. The idea of generalizing classical probability theory by putting the emphasis on the ring of random variables rather than the underlying probability space is not new (it is used, e.g., in free probability), however, we are not aware of another source that develops the material completely free of any analytic notions (i.e. for algebras of random variables over an arbitrary ring instead of $\RR$). 

Let $R$ be a ring.

\begin{definition} An \emph{algebraic $R$-probability measure}\footnote{Note that there is no actual measure in the classical sense here; instead we are thinking of a measure as equivalent to the corresponding integration functional, which, for a probability space, is the expectation.} $\mu$ on an $R$-algebra $A$ with values in an $R$-algebra $A'$ is an $R$-module map
\[ \EE_\mu: A \rightarrow A' \]
sending $1_{A}$ to $1_{A'}$. 
\end{definition}

In this setting, we will refer to elements of $A$ as \emph{random variables} and to the map $\EE_\mu$ as the \emph{expectation}. When the measure is implicit, we will sometimes write $\EE$ without the subscript. 

\begin{example}
If $(Y,\mu)$ is a finite probability space in the classical sense (i.e. $Y$ is a finite set and $\mu$ is a real measure on $X$ with $\mu(X)=1$), then we obtain an algebraic $\RR$-valued probability measure on $\mathrm{Map}(Y, \RR)$ with values in $\RR$ sending a random variable 
\[ X \in \mathrm{Map}(Y, \RR) \]
to
\[ \EE_\mu[X] := \int_\mu X = \sum_{y \in Y} X(y) \mu( \left\{ y \right\} ). \] 
\end{example}

\begin{example}\label{example:GrothProb}
If $S/\KK$ is a variety with $[S]$ invertible in $\Mc_\LL$, then we obtain an algebraic $K_0(\Var/\KK)$-probability measure on $K_0(\Var/S)$ with values in $\Mc_\LL$ (cf. Section \ref{sec:GrothendieckRing} for the notation on Grothendieck rings) such that, for $Y/S$ quasi-projective, 
\[ \EE_\mu\left[ [Y/S] \right ] \mapsto \frac{[Y]}{[S]}. \]

If $\KK$ is a finite field $\FF_q$, then for any $f$ this measure specializes to the classical uniform probability measure on the finite set $S(\FF_{q})$: for $Y/S$, the random variable $[Y/S]$ specializes to the random variable on $S(\FF_q)$ given by
\[ s \mapsto \# Y_s(\FF_q).\]
Thus this measure gives a natural motivic lift of the uniform measure on the set of points of a variety.  

If $\KK$ is of characteristic zero, then for any variety $[Y/S]$, we obtain a map 
\begin{align*}
\Lambda & \rightarrow  K_0(\Var / S) \\
f & \mapsto  (f, [Y/S]) 
\end{align*}
and, if we pull-back the measure via this map, we obtain an algebraic $\ZZ$-probability measure on $\Lambda$ with values in $\Mc_\LL$. We will often use this construction when we discuss stabilization (with $\Mc_\LL$ replaced with $\motcomp$ or its image under a motivic measure). We think of this as a type of uniform probability measure, but with values in a complicated ring that often remembers more subtle information.
\end{example}

\begin{definition} A set $\{ a_i\}_{i\in I}$ of elements $a_i \in A$ is \emph{independent} if, for any finite subset $\{i_1,...,i_l\} \in I$ and $k_1,.., k_l \in \ZZ_{\geq1}$,
\[ \EE[a_{i_1}^{k_1} \cdot ... \cdot a_{i_l}^{k_l}] = \EE[a_{i_1}^{k_l}]\cdot ... \cdot \EE[a_{i_l}^{k_l}] \]
\end{definition}

\begin{definition} A sequence of probability measures $\{\mu_j\}$ on $A$ with values in $A'$ a separated topological ring \emph{converges} to a measure $\mu_\infty$ if, for every $a \in A$,
\[ \lim_{j \rightarrow \infty} \EE_{\mu_j} [ a ] = \EE_{\mu_\infty}[a]. \]
\end{definition}

In this setting, we will be concerned with random variables that may not be independent for any of the measures $\mu_j$, but behave as independent random variables in the limit. 

\begin{definition} If $\{\mu_j\}$ is a sequence of probability measures on $A$ with values in $A'$ a complete topological ring, a subset $\{ a_i\}_{i\in I}$ of elements $a_i \in A$ is \emph{asymptotically independent} if for any finite subset $\{i_1,...,i_l\} \in I$ and $k_1,.., k_l \in \ZZ_{\geq1}$,
\[ \lim_{j\rightarrow \infty} \EE_{\mu_j}[a_{i_1}^{k_1} \cdot ... \cdot a_{i_l}^{k_l}] = \left( \lim_{j \rightarrow \infty} \EE_{\mu_j}[a_{i_1}^{k_l}] \right) \cdot ... \cdot \left( \lim_{j \rightarrow \infty} \EE[a_{i_l}^{k_l}] \right)\]
and all of these limits exist. 
\end{definition}

In particular, if we have a sequence of measures $\mu_j$ on the polynomial ring 
\[ R[x_1, x_2,...], \] 
and the variables $x_i$ are asymptotically independent then the measures $\mu_j$ converge to a $\mu_\infty$, and for any 
\[ g \in R[x_1, x_2,...], \]
$\EE_{\mu_\infty}[g]$ can be expressed in terms of the moments of the $x_i$ by writing $g$ as a sum of monomials and applying asymptotic independence.  

Just as we understand independence in terms of joint moments, we will also understand distributions in terms of moment generating functions. Because our applications all involve binomial and Bernoulli random variables, we will use falling moment generating functions. Because we need denominators for our moment generating functions, when we discuss these types of random variables we will assume that $R$ is a $\QQ$-algebra. 

\begin{definition} \hfill
\begin{enumerate}
\item A \emph{Bernoulli random variable} with probability $p \in A'$ is a random variable $a \in A$ such that (for $t$ a formal variable)
\[ \EE[(1+t)^a] = 1+ p t. \]
\item A \emph{binomial random variable}, the ``sum of'' $s \in A'$ independent Bernoulli random variables with probability $p \in A'$, is a random variable $a\in A$ such that (for $t$ a formal variable)
\[ \EE[(1+t)^a]=(1+p t)^s. \]
\end{enumerate}
\end{definition}

\begin{remark}\label{rem:multinomialMoments}
For a binomial random variable $a \in A$, the ``sum of'' $s \in A'$ independent Bernoulli random variables with probability $p \in A'$, we have for any $m$ and $t_1,...,t_m$ formal variables, 
\[ \EE[(1+t_1 + t_2 + +... +t_m)^a]=\left(1+p(t_1 +... + t_m)\right)^s, \]
which can be seen by expressing the multinomial coefficients of the exponents appearing on each side as integer multiples of binomial coefficients of the exponents. We will use this fact later in the proofs of our main theorems. 
\end{remark}

\section{Proofs of Theorem \ref{thm:VirtualLocHS} and Corollary \ref{cor:RelativeStab}}\label{sec:Proof}

In this section, we state and prove the general versions of Theorem \ref{thm:VirtualLocHS} and Corollary \ref{cor:RelativeStab} alluded to in Remark \ref{rmk:ChenGetzler}. 

\subsection{Motivic analogs of local systems attached to character polynomials}\label{subsec:MotAnalogCharPoly}
Let $Y/\CC$ be a smooth quasi-projective variety. We begin by constructing a motivic analogs of the local system on $\Conf^n (Y)$ attached to a character polynomial ${p \in \QQ[X_1,X_2,...]}$. 

By Lemma \ref{lem:Leray}, the following diagram commutes
\[ \xymatrix{
	K_0(S_n-\Set) \ar[r]\ar[dd]	& K_0(\FinEt / \Conf^n (Y)) \ar[dd]_{[f:Z \rightarrow \Conf^n (Y)] \mapsto [f_*\QQ]} \ar[r] & K_0(\Var / \Conf^n (Y)) \ar[d]^{[Z/\Conf^{n}(Y)] \mapsto [Z]_\HS} \\
	&			& K_0(\HS) \\
	K_0(\Rep S_n) \ar[r] &	 K_0(\GVSH / \Conf^n (Y)) \ar[ur]_{\;\;\;\;\;\;\; [\Vc] \mapsto \chi_\HS([\Vc])} & 
		} \]
Here, the two horizontal maps from the leftmost column are coming from the $S_n$-cover $\PConf^n (Y) \rightarrow \Conf^n (Y)$. All of the maps except those to $K_0(\HS)$ are maps of pre-$\lambda$ rings. 

If we denote by $V_n$ the permutation representation on $\{1,...,n\}$, the map from the ring $\QQ[X_1,X_2,...]$ of character polynomials to $\Rep S_n \otimes \QQ$ is given by identifying $\QQ[X_1, X_2,...]$ with $\Lambda_\QQ$ via $X_i \mapsto p_i'$ and then sending a character polynomial $p$ to $(p, [V_n]).$ 

On the other hand, the class of the set $\{1,...,n\}$ in $K_0(S_n-\Set)$ maps via the top horizontal arrows to the class $[Z_n / \Conf^n (Y)]$ -- recall that $Z_n$ is the universal configuration of $n$ distinct points 
\[ Z_n =\{ (y, c) \in Y \times \Conf^n (Y) \, | \, y\in c \}. \]

We denote by $\alpha$ the forgetful map from $K_0(\Var/\Conf^n (Y))$ to $K_0(\Var)$. In the notation of Theorem \ref{thm:VirtualLocHS}, we obtain
\[ \chi_\HS(\Vc_{\pi_{p,n}}) = \alpha\left((p, [Z_n / \Conf^n (Y)]) \right)_\HS. \]

Thus, 
\[ [p]_n := (p, [Z_n / \Conf^n (Y)]) \in K_0(\Var/\Conf^n (Y)) \]
provides a natural motivic analog of $\Vc_{\pi_{p,n}}$. Note that we can define these motivic analogs without the smoothness condition on $Y$, which we use only in relating back to the theory of geometric variations of Hodge structure as in \cite{Arapura-LerayMotivic}. 

\subsection{General versions of Theorem \ref{thm:VirtualLocHS} and Corollary \ref{cor:RelativeStab}}
We will want to consider a motivic measure $\phi$ valued in a ring $R$ such that 
\[ \lim_{n \rightarrow \infty} \frac{[\Conf^n (Y)]_\phi}{\LL^{n \dim Y}_\phi} \]
exists and is invertible. This is equivalent to $Y$ satisfying the property $\MSSP_\phi$ of Vakil and Wood \cite{VakilWood-Discriminants} plus the limit of (normalized) symmetric powers being invertible; we denote this condition by $\MSSP_\phi^*$. 

\begin{remark}\label{rmk:HS-MSSP} If $\phi$ is the Hodge measure of Example \ref{example:mot-measure}, then every $Y$ satisfies $\MSSP_\phi^*$ (cf. \cite[1.26-(ii)]{VakilWood-Discriminants}).
\end{remark}

If $\phi$ is a motivic measure such that $Y$ satisfies $\MSSP_\phi^*$ then, for $n$ sufficiently large, we can define an algebraic probability measure $\EE_{\phi,n}$ on $\Lambda$ with values in $R_\QQ$ by
\[ \EE_{\phi,n}[ p ] = \alpha( [p]_n )_\phi / [\Conf^n (Y)]_\phi \]
where as before $\alpha$ denotes the forgetful map from $K_0(\Var/\Conf^n (Y))$ to $K_0(\Var)$. This is the pullback of the uniform measure on $\Conf^n (Y)$ as in Example \ref{example:GrothProb} by the map $\Lambda \rightarrow K_0(\Var / \Conf^n (Y))$ sending $p$ to $(p, [Z_n])$. 

We denote by 
\[ M_k([Y]) = (p_k', [Y]) \]
the exponent of $(1+t^k)^{-1}$ in the naive Euler product\footnote{To our knowledge, this motivic Euler product was first studied by Bourqui~\cite[2.2]{Bourqui-ProduitEulerien}.} for the Kapranov zeta function of $Y$ (cf. Subsection \ref{subsec:GrothPow}). $M_k([Y])$  should be viewed as a motivic analog of the ``closed points of degree $k$'' on a variety over a finite field. 

Using the language of Section \ref{sec:Probability}, the general version of Theorem \ref{thm:VirtualLocHS} is  

\begin{theorem}\label{thm:GeneralMotivicStab}
Let $Y/\CC$ be a quasi-projective variety, and suppose that $Y$ satisfies $\MSSP_\phi^*$. Then, with respect to the sequence of probability measures $\EE_{\phi,n}$ on the ring of character polynomials $\QQ[X_1,X_2,...]$ given above, the character polynomials $X_i$ are asymptotically independent random variables. Furthermore, $X_k$ is asymptotically a binomial random variable, the ``sum of'' $M_k([Y])_\phi$ independent Bernoulli random variables that are $1$ with probability 
\[ \frac{ 1 }{1 + \LL^{k \cdot \dim Y}_\phi}. \]
\end{theorem}

The general version of Corollary \ref{cor:RelativeStab} is

\begin{corollary}\label{cor:GeneralRelativeStab}
In the setting of Theorem \ref{thm:GeneralMotivicStab}, if $\tau = a_1^{l_1}...a_m^{l_m}$ is a partition then in $R_\QQ$,
\begin{multline*}
\lim_{n \rightarrow \infty} \frac{  [\Conf^{\tau \cdot *^{n-|\tau|}} (Y)]_\phi }{ [\Conf^n (Y)]_\phi} = \\
\frac{\partial_{t}^{\tau}|_{t_\bullet = 0}}{\tau!} \prod_k \left( 1 + \frac{1}{1+\LL^{k\dim Y}_\phi} ( t_1^k + t_2^k + ... + t_m^k) \right)^{ M_k([Y])_\phi }. 
\end{multline*}
\end{corollary}

\begin{example}\label{example:stably-rational}
If $\phi$ is the natural map to $\motcomp$ and $Y$ is stably rational, then by \cite[1.26(i)]{VakilWood-Discriminants}, $Y$ satisfies $\MSSP_\phi$. In fact, $Y$ satisfies $\MSSP_\phi^*$ since the class of any stably rational variety is invertible in $\motcomp$ (it differs from some $\LL^m$ by something lying lower in the dimension filtration). Thus, Theorem \ref{thm:GeneralMotivicStab} and Corollary \ref{cor:GeneralRelativeStab} hold in this case. In particular, because $\phi$ does not factor through Chow motives, this is an application of Theorem \ref{thm:GeneralMotivicStab} that does not also follow from the work of Getzler as described in Appendix \ref{appendix:GetzlerChen}.  
\end{example}

Theorem \ref{thm:VirtualLocHS} and Corollary \ref{cor:RelativeStab} follow from Theorem \ref{thm:GeneralMotivicStab} and Corollary \ref{cor:GeneralRelativeStab} by taking $\phi$ to be the Hodge measure, for which $\MSSP^*_\phi$ holds by Remark \ref{rmk:HS-MSSP}, and using the discussion of Subsection \ref{subsec:MotAnalogCharPoly} to see 
\[ \EE_{\HS, n}[p] =\frac{ \chi_\HS(\Vc_{\pi_{p,n}}) }{[\Conf^n (Y)]_\HS}. \]

\subsection{Proofs of Theorem \ref{thm:GeneralMotivicStab} and Corollary \ref{cor:GeneralRelativeStab}}

\begin{proof}[Proof that Theorem \ref{thm:GeneralMotivicStab} is equivalent to Corollary \ref{cor:GeneralRelativeStab}.]

Using the geometric description of the power structure, we have (similar to Example \ref{example:ConfGen}),
\[ (1 + t_1 + ... + t_m)^{_\Pow [Z_n / \Conf^n (Y)]} = \sum_{\tau=a_1^{l_1}...a_m^{l_m}} \left[\Conf^{\tau}\left( Z_n / \Conf^n (Y) \right) / \Conf^n(Y) \right]  t^\tau. \]
By Lemma \ref{lem:EulerCoeffs}, the left-hand side is also equal in $K_0(\Var/\Conf^n (Y))_\QQ$ to the infinite product
\[ \prod_{k=1}^\infty (1 + t_1^k + ... + t_m^k)^{(p_k', [Z_n / \Conf^n (Y)])}. \]

Observe that the relative configuration space $\Conf^{\tau}\left( Z_n / \Conf^n (Y) \right)$ is isomorphic, as a variety over $\CC$, to 
$\Conf^{\tau \cdot *^{n-|\tau|}} (Y).$ Thus, by applying the forgetful map $\alpha$ to the coefficients of these power series, we find
\[ \alpha\left( \prod_{k=1}^\infty (1 + t_1^k + ... + t_m^k)^{(p_k', [Z_n / \Conf^n (Y)])} \right) = \sum_\tau \left[\Conf^{\tau \cdot *^{n-|\tau|}}(Y)\right] t^\tau. \]
in $K_0(\Var)_\QQ$. In particular, if we fix a $\tau$, we find
\[ \alpha \left( \frac{\partial_{t}^{\tau}|_{t_\bullet = 0}}{\tau!}  \prod_{k=1}^\infty (1 + t_1^k + ... + t_m^k)^{(p_k', [Z_n / \Conf^n (Y)])} \right) = \left[\Conf^{\tau \cdot *^{n-|\tau|}}(Y)\right]. \]
Applying $\phi$ and dividing both sides by $[\Conf^n (Y)]_\phi$ we obtain 
\[ \EE_{\phi}\left[\frac{\partial_{t}^{\tau}|_{t_\bullet = 0}}{\tau!} \prod_{k=1}^\infty (1 + t_1^k + ... + t_m^k)^{(p_k', [Z_n / \Conf^n (Y)])} \right] = \frac{\left[\Conf^{\tau \cdot *^{n-|\tau|}}(Y)\right]_\phi}{[\Conf^n (Y) ]_\phi}. \]

In particular, if Theorem \ref{thm:GeneralMotivicStab} is satisfied, then taking the limit in $n$, we can move the expectation inside the product and use the moments of a binomial distribution as in Remark \ref{rem:multinomialMoments} to obtain
\[ \frac{\partial_{t}^{\tau}|_{t_\bullet = 0}}{\tau!} \prod_{k=1}^\infty \left(1 + \frac{1}{1+\LL^{k \cdot \dim Y }_\phi}(t_1^k + ... + t_m^k)\right)^{(p_k', [Y])} = \lim_{n \rightarrow \infty} \frac{\left[\Conf^{\tau \cdot *^{n-|\tau|}}(Y)\right]_\phi}{[\Conf^n (Y)]_\phi}. \] 
Thus Theorem \ref{thm:GeneralMotivicStab} implies Corollary \ref{cor:GeneralRelativeStab}.

On the other hand, the $c_\tau$ are a basis for $\Lambda_\QQ$, thus there is a unique map of $\QQ$-vector spaces
\[ \Lambda_\QQ \rightarrow R_\QQ \]
sending $c_\tau$ to 
\[ \frac{\partial_{t}^{\tau}|_{t_\bullet = 0}}{\tau!} \prod_{k=1}^\infty \left(1 + \frac{1}{1+\LL^{k\cdot \dim Y}}(t_1^k + ... + t_m^k)\right)^{\phi((p_k', [Y]))} . \]

Using the formula 
\[ c_\tau =  \frac{\partial_{t}^{\tau}|_{t_\bullet = 0}}{\tau!} \prod_{k=1}^\infty (1 + t_1^k + ... + t_m^k)^{p_k'} \]
of Lemma \ref{lem:conf-Euler}, we see that this is also the unique map such that
\[ \prod_{i=1}^{m} \binom{ p'_k }{ l_k } \mapsto
 \prod_{i=1}^{m} \binom { \phi\left( (p_k', [Y])\right) }{ l_k} \left( \frac{1}{1+\LL^{k \cdot \dim Y }_\phi}\right)^{l_k}. \]
(i.e. the map obtained by declaring the $p_k'$ to be independent with the given binomial distributions). 

If we assume Corollary \ref{cor:GeneralRelativeStab}, then by computing with the basis $c_\tau$, we find the map 
\begin{align*} \Lambda_\QQ & \rightarrow R_\QQ \\
g & \mapsto \lim_{n \rightarrow \infty} \EE_{\phi, n}[g] = \lim_{n \rightarrow \infty} \EE_\phi[ (g, [Z_n/\Conf^n (Y)])  ]
\end{align*}
is defined, and equal to map described above, so that Corollary \ref{cor:GeneralRelativeStab} implies Theorem~\ref{thm:GeneralMotivicStab}. 

\end{proof}

\begin{proof}[Proof of Corollary \ref{cor:GeneralRelativeStab}.]

By the description of the power structure on $K_0(\Var)$ in Subsection \ref{subsec:GrothPow}, we have
\[ (1 + t_1 s  + ... + t_m s + s)^{_\Pow [Y]} = \sum_{(l_1,...,l_m,n) \in {\ZZ_{\geq0}^{m+1}}} \left[\Conf^{a_1^{l_1}a_2^{l_2}...a_m^{l_n}\cdot *^n} (Y)\right] t^{l_1}... t^{l_m} s^{n+l_1 +... +l_m}. \]

By Lemma \ref{lem:EulerCoeffs}, this is equal to
\[ \prod_{k=1}^{\infty} (1 + t_1^k s^k + ... + t_m^k s^k + s^k)^{(p_k', [Y])}. \]

Thus, if we fix a partition $\tau=a_1^{l_1}...a_m^{l_m}$, we obtain
\[ \frac{\partial_t^\tau}{\tau!}|_{t_\bullet=0} \prod_k (1 + t_1^k s^k + ... + t_m^k s^k + s^k)^{(p_k', [Y])} = \sum_{n \in \ZZ_\geq 0} \left[\Conf^{\tau \cdot *^n} (Y)\right] s^{|\tau| + n}. \]

Because the differentiation on the left-hand side is with respect to the $t$ variables, we can pull out a factor of
\[ \prod_k (1+s^k)^{(p_k', [Y])} = (1+s)^{_\Pow [Y]} = \frac{(1-s)^{_\Pow -[Y]}}{(1-s^2)^{_\Pow -[Y]}} = \frac{ Z_Y(s)}{ Z_Y(s^2)} \] 
to obtain 
\[ \frac{\partial_t^\tau}{\tau!}|_{t_\bullet=0} \prod_k \left( 1 + \frac{s^k}{1+s^{k}}(t_1^k + ... + t_m^k) \right)^{(p_k', [Y])} \frac{Z_Y(s)}{Z_Y(s^2)} = \sum_{n \in \ZZ_\geq 0} \left[\Conf^{\tau \cdot *^n}(Y)\right] s^{|\tau|+n}. \]

We claim that
\begin{equation}\label{equation:rat-func} E(s):=\frac{\partial_t^\tau}{\tau!}|_{t_\bullet=0} \prod_k \left( 1 + \frac{s^k}{1+s^{k}}(t_1^k + ... + t_m^k) \right)^{(p_k', [Y])} \end{equation}
is a rational function of $s$ which can be written as a finite sum of fractions where the numerators are in 
\[ \QQ[(p_1', [Y]), ..., (p_{|\tau|}', [Y])], \]
and the denominator is a product of factors of the form 
\[ \frac{s^k}{1+s^{k}}. \]
To see this, expand each term with binomial coefficients then expand the product and use that the differentiation is picking out the coefficient of $t^\tau$. 

In particular, $E(\LL^{-\dim Y})$ exists already in $\motcomp \otimes \QQ$, and, applying \cite[Lemma 5.4]{VakilWood-Discriminants}, we see that
\[ \lim_{n \rightarrow \infty} \frac{\left[\Conf^{\tau\cdot*^n} (Y)\right]_\phi}{[\Conf^{|\tau|+n} (Y)]_\phi} = E(\LL^{-\dim Y})_\phi.\]
Evaluating the right-hand side using the expression (\ref{equation:rat-func}) then gives the desired result. 
\end{proof}

\appendix

\section{Work of Chen and Getzler}\label{appendix:GetzlerChen}
In this appendix we show that Theorem \ref{thm:VirtualLocHS} follows from work of Getzler \cite{Getzler-MixedHodge}. In fact, using \cite{Getzler-MixedHodge} we deduce the stronger Theorem \ref{thm:CharPolyGenFunc} below, which is a Hodge-theoretic analog of a result of Chen \cite[Theorem 3]{Chen-ConfPointCounting} for point counting over finite fields (in Theorem \ref{thm:Chen} below we state a version of this result with the notation and the setting modified slightly from \cite{Chen-ConfPointCounting}). From Theorem \ref{thm:CharPolyGenFunc} we can deduce Theorem \ref{thm:VirtualLocHS} using a lemma of Vakil-Wood \cite{VakilWood-Discriminants} (this step is very similar to Chen's deduction of \cite[Corollary 4]{Chen-ConfPointCounting} from \cite[Theorem 3]{Chen-ConfPointCounting}). We note that, as in \cite{Getzler-MixedHodge}, instead of working with the Hodge measure, we could prove the more general Theorem \ref{thm:GeneralMotivicStab} for any $\phi$ factoring through Chow motives.  

For $\overline{l}=(l_1,...l_k)$ a sequence of non-negative integers we define the character polynomial 
\[ \binom{X}{\overline{l}} := \prod_{i=1}^m \binom{X_i}{l_i} \]
where $X_i$ as before is the function counting cycles of length $i$. For a permutation $\sigma$ in of a set $A$,  
\[ \binom{X}{\overline{l}}(\sigma) \]
is the number of subsets consisting of $\sum {i \cdot l_i}$ elements of $A$ on which $\sigma$ acts as a permutation of cycle type $\overline{l}$ (i.e. has $l_1$ 1-cycles, $l_2$ 2-cycles, ...). For $Y$ a variety over $\FF_q$ and $c$ a configuration of $n$ points in $\Conf^n (Y)(\FF_q)$ we denote by $\sigma_c$ the permutation given by $\Frob_q$ acting on the $n$ geometric points of $c$. 

\begin{theorem}[Chen \cite{Chen-ConfPointCounting}]
\label{thm:Chen}
Let $Y$ be a smooth connected quasi-projective variety over $\FF_q$ and define 
\[ M_k(Y,q) := \sum_{d|k} \frac{1}{d} \mu(k/d) |Y(\FF_{q^d})| \]
(the number of closed points of degree $k$ on $Y$). For any sequence of non-negative integers $\overline{l}=(l_1,...,l_m)$, 
\[ \sum_{n=0}^\infty \left [\sum_{c \in \Conf^n (Y)(\FF_q)} \binom{X}{\overline{l}}(\sigma_c) \right] t^n = \frac{Z(Y,t)}{Z(Y,t^2)}\cdot \prod_{k=1}^m \binom {M_k(Y,q)}{l_k} \left(\frac{t^k}{1+t^k}\right)^{l_k}. \]
\end{theorem}

We show 

\begin{theorem}\label{thm:CharPolyGenFunc}
Let $Y/\CC$ be a smooth connected quasi-projective variety. For any sequence of non-negative integers $\overline{l}=(l_1,...,l_m)$, 
\[ \sum_{n=0}^\infty \chi_\HS\left(\Vc_{\pi_{\binom{X}{\overline{l}},n}}\right) t^n = \frac{Z_{Y,\HS}(t)}{Z_{Y,\HS}(t^2)} \cdot \prod_{k=1}^m \binom{M_k([Y]_\HS)}{l_k} \left(\frac{t^k}{1+t^k}\right)^{l_k}. \] 
\end{theorem}
\begin{remark}
Theorem \ref{thm:CharPolyGenFunc} is a Hodge-theoretic analog of Theorem \ref{thm:Chen}, as can be seen by interpreting the left-hand side of the formula in Theorem \ref{thm:Chen} using the Grothendieck-Lefschetz formula as the trace of Frobenius acting on the virtual $l$-adic local system attached to the character polynomial $\binom{X}{\overline{l}}.$  
\end{remark}

\begin{proof}
By \cite[Theorem 5.6 and formula (2.6)]{Getzler-MixedHodge} and the discussion of Subsection \ref{subsec:MotAnalogCharPoly}, we see that the coefficient of $t^n$ on the left hand side is equal to the inner product of the class function on $S_n$ 
\[ \binom{X}{\overline{l}}, \]
viewed as a degree $n$ symmetric function in $\Lambda$, with the degree $n$ component of 
\[ \prod_{k=1}^\infty  (1+p_k)^{M_k([Y]_\HS)}.\]
The class function attached to $\binom{X}{\overline{l}}$, viewed as a degree $n$ symmetric function, is equal to
\[ \sum_{\tau \vdash n} \binom{X}{\overline{l}}(\tau) \cdot p_\tau/z_\tau. \]
where partitions $\tau$ are thought of as conjugacy classes in $S_n$ and $z_\tau$ is the order of the centralizer of an element of the conjugacy class $\tau$. 

Because the $p_\tau$ are an orthogonal basis with $\langle p_\tau, p_\tau/z_\tau \rangle=1$, we find
that the left hand side is equal to 
\[ \frac{\partial}{\partial z_1}^{l_1} ... \frac{\partial}{\partial z_m}^{l_m} \prod_{k=1}^\infty  (1+ t^k \cdot z_k)^{M_k(Y)} |_{z_1 = z_2 = ... =z_m=1}. \]
We can rewrite this as 
\[ \left(\prod_{k=1}^\infty (1 + t^k)^{M_k(V)}\right) \cdot \prod_{k=1}^m \binom{M_k(Y)}{l_k} 
\left(\frac{t^{k}}{1+t^k}\right)^ {l_k} \]

The zeta factors come from the left-hand side using the identity 
\[ (1+t^k)=(1-t^k)^{-1}/(1-t^{2k})^{-1} \]
and the fact that the $M_k([Y]_\HS)$ are \emph{defined} as the exponents of the Euler product of $Z_{Y,\HS}(t)$. The identification of the zeta quotient with the series of configuration spaces is standard, and holds already in the Grothendieck ring of varieties. 
\end{proof}

Theorem \ref{thm:VirtualLocHS} follows from Theorem \ref{thm:CharPolyGenFunc} by applying \cite[Lemma 5.4]{VakilWood-Discriminants}.  

\bibliography{references}
\bibliographystyle{plain}

\end{document}